\documentclass[11pt,reqno]{amsart}
\usepackage{geometry}                
\usepackage{graphicx}
\usepackage{amssymb}
\usepackage{epstopdf}
\usepackage{amsmath}   
\usepackage{amsthm}
\usepackage{amsfonts}
\usepackage{latexsym}
\usepackage{color}
\theoremstyle{plain}
\newtheorem{thm}{Theorem}
\newtheorem{prop}[thm]{Proposition}
\newtheorem{lem}[thm]{Lemma}
\numberwithin{equation}{section}
\newtheorem{cor}[thm]{Corollary}

\theoremstyle{remark}
\newtheorem{rek}[thm]{Remark}

\theoremstyle{definition}
\newtheorem{defi}[thm]{Definition}

\DeclareMathOperator*{\hocolim}{hocolim }
\DeclareMathOperator{\Sym}{Sym}

\usepackage{tikz-cd}

 \tikzset{commutative diagrams/.cd,
mysymbol/.style={start anchor=center,end anchor=center,draw=none}
}
\newcommand\MySymb[2][\bigstar]{%
\arrow[mysymbol]{#2}[description]{#1}}
  
\usepackage[mathscr]{euscript}

\makeatletter
\def\@xfootnote[#1]{%
  \protected@xdef\@thefnmark{#1}%
  \@footnotemark\@footnotetext}
\makeatother

\title{A proof of the Dold--Thom theorem via factorization homology}
\author{Lauren Bandklayder }\thanks{\noindent This work was completed with the partial support of an NSF Graduate Research Fellowship.}

\begin{document}

\newpage
\maketitle

\begin{abstract}
We give a new proof of the classical Dold--Thom theorem by using factorization homology.  Our method is new, quick, and more direct, avoiding the Eilenberg--Steenrod axioms entirely and, in particular, making no use of the theory of quasi-fibrations.
\end{abstract}

\section{Introduction}
The Dold--Thom theorem is a classical result giving a beautiful relation between homotopy and homology. It states that for a nice, based topological space, $M$, there is an isomorphism between the homotopy groups of the infinite symmetric product of $M$ and the homology groups of $M$ itself: $$\pi_*(\mathrm{Sym}(M;A)) \cong \widetilde{H}_*(M;A).$$

 A fundamental result with applications in both algebraic topology and algebraic geometry, this theorem has received much attention since it was first published in 1958. In 1959, Spanier used the equivalence of homology theories exhibited to understand Spanier-Whitehead duality \cite{Spanier}.  Later, McCord gave a convenient model for the categorical tensor of a based space and a topological abelian group which generalized Dold and Thom's infinite symmetric product \cite{MCMC}. Inspired in part by McCord, the notes of Floyd and Floyd \cite{Floyd} discuss the infinite symmetric as a source of models for spectra.  Both McCord's model and the notes \cite{Floyd} are treated in the paper of Kuhn, \cite{Kuhn}, where more historical perspective is also given. Segal, too, gave a generalized model for the infinite symmetric product in \cite{Segal}, viewing it as a labelled configuration space.  In 1996, Gajer gave an intersection-homology variant of the Dold--Thom theorem \cite{Gaj}.  An equivariant formulation of the theorem was given by dos Santos in \cite{DS}, which generalizes an equivariant integral-coefficient formulation given by Lima-Filho in \cite{Lima}.  More recently, Suslin and Voevodsky used the theorem to define motivic cohomology, giving a means translating techniques and results of algebraic topology into algebraic geometry \cite{Voe}.  As giving a complete history of the Dold--Thom theorem would constitute a paper of its own, we have only outlined some of the highlights here, but we hope this convinces the reader of its rich history and utility.

The original proof of the Dold--Thom theorem proceeds by verifying that the composition of functors $\pi_*(\Sym(-;A))$ satisfies the Eilenberg--Steenrod axioms for a reduced homology theory \cite{ES} and agrees with singular homology when applied to $S^0$. The crux of this, and most other known proofs, is to check a certain map is a \emph{quasi-fibration} \cite{DTqf}. This is a bit of a technical digression, and it is our goal in this note to present a more direct proof which does not require any such fact, nor does it appeal to the Eilenberg--Steenrod axioms at all. For simplicity we will restrict ourselves to the case where $M$ is a smooth manifold, but all results carry over with only a slight modification to the case where $M$ is a topological manifold.

The rough outline of our proof is as follows. The heart lies in Proposition \ref{main}, where we use a hypercover argument to see that for $\mathrm{Disk}_{*/M}$ an appropriate category of disks disjointly embedded in $M$, there is a homotopy equivalence between the infinite symmetric product of $M$ and the factorization homology of $M$: $$\mathrm{Sym}(M;A) \simeq \underset{U \in \mathrm{Disk}_{*/M}}{\hocolim  } (\mathrm{Sym}(U; A)).$$ 
\noindent In Lemma \ref{john} we will show how to use the Dold--Kan correspondence to pass to chain complexes. It is crucial to this step that we take the homotopy colimit over a \emph{sifted} category, and we wish to emphasize this point to the reader, as it is precisely this fact that allows us to analyze this same homotopy colimit in spaces as one in chain complexes. Once in the context of chain complexes, we finish the proof by showing the chain complex $\underset{U \in \mathrm{Disk}_{*/M}}{\mathrm{hocolim}} (\Sym(U;A))$ is quasi-isomorphic to the reduced singular chain complex, $\widetilde{C}_*(M;A)$, which precisely implies the desired Dold--Thom isomorphisms: $\pi_i(\Sym(M;A) \cong \widetilde{H}_i(M;A)$.

\subsection*{Acknowledgements} It is an honor to first and foremost thank my advisor, John Francis, for suggesting this approach, as well as for his patience, guidance, and support. I am also filled with gratitude towards Ben Knudsen and Dylan Wilson, both for very many helpful conversations on this work and for commentary on an earlier draft. I would like thank Paul VanKoughnett and Jeremy Mann for reading an earlier draft and offering useful feedback, and Elden Elmanto for helpful conversations, and for his enthusiasm and encouragement throughout. I am grateful to Nick Kuhn for pointing me to the references \cite{Floyd, Kuhn} and for enlightening comments on an earlier draft. Finally, this paper was written while partially supported by an NSF Graduate Research Fellowship, and I am very grateful for their support. 
 
\vspace{5mm}
\section{Preliminaries}
\vspace{3mm}

\subsection{The infinite symmetric product}
In this section we briefly recall the definition of the infinite symmetric product associated to a topological space, motivated by the idea of defining the free topological $A$-module on a space, $M$.
\newpage
\begin{defi} \label{symdef} For $M$ a pointed space with basepoint, $*$, and $A$ an abelian group with identity, $e$, define the \textbf{infinite symmetric product of M with coefficients in A} to be the space $$\mathrm{Sym}(M;A) := \{  (S,l) :  *\in S\subset M,\ |S|<\infty,\text{ and } l:(S\backslash*) \rightarrow A  \} / \sim $$ 
where $\sim$ is the equivalence relation given by $(S,l) \sim (S \cup \{x\}, l')$ when $l'$ is the map which agrees with $l$ on $S$ and sends $x$ to $e$, topologized with the finest topology making the following maps continuous for any finite set $I$:
\begin{align}
f_{I}: M^{I_+} \times A^I &\rightarrow \mathrm{Sym}(M;A) \nonumber \\
(c: I_+ \rightarrow M, l: I \rightarrow A) &\mapsto \Big[(c(I_+), l': s \mapsto \sum_{i \in c^{-1}(s)} l(i))\Big] . \nonumber
\end{align}
Here $M^{I_+}$ denotes based maps $I_+ \rightarrow M$, $A^I$ denotes all maps $I \rightarrow A$, and both are endowed with the product topology. 
\end{defi}

That is, taking the union over all finite $I$ of the maps $f_I$ gives a surjection \\$M^{I_+}\times A^I \twoheadrightarrow \mathrm{Sym}(M;A)$, and we say $U \subset \mathrm{Sym}(M;A)$ is open if and only if its inverse image is open under each of these maps. In particular, this topology requires that labels vanish at the basepoint and that points labeled by the identity be forgotten, and allows for points to collide whence their labels add.  

\vspace{3mm}

\begin{lem} \label{finite} For $I$ a finite set, there is a homeomorphism $$\mathrm{Sym}(I_+;A) \cong \widetilde{H}_0(I_+;A)= A^I$$ where we consider $A$ as a topological space with the discrete topology and $A^I$ with the product topology.
\end{lem}

\begin{proof}
Let $I_+$ equal $\{+,1,...,n\}$ and let $e$ denote the identity of $A$.  We can define the following map, $f$, from $\mathrm{Sym}(I_+;A)$ to $A^I$:
\begin{align}
f: \mathrm{Sym}(I_+;A) &\rightarrow A^I \nonumber \\
\big[ (S,l) \big] &\mapsto (l_1(1),l_2(2),...,l_n(n))\nonumber
\end{align}
where for each $i$ in $I$, we define $l_i$ as follows:
$$l_i(i):= \begin{cases}
l(i) & i \in S \\
e & else.
\end{cases}$$ 
This is well defined:
suppose $\big[ (S,l) \big] =\big[(S',l') \big]$ so without loss of generality we can assume that $S'$ is equal to $S\amalg j$, and that $l'$ is identical to $l$ on $S$, and sends $j$ to $e$. Note that these assumptions imply:
\begin{enumerate}
\item{for each $i$ in $S \subset S'$, we have the equalities $l'_i(i)=l'(i)=l(i)=l_i(i),$ }
\item{for each $i$ not in $S'$, $i$ is also not in $S$ and so we have the equalities $l'_i(i)=e=l_i(i),$}
\item{since $j$ is in $S'$ but not $S$, we have the equalities $l'_j(j)=l'(j)=e=l_j(j).$ }
\end{enumerate}
These observations prove that $f(\big[ (S,l) \big])$ is equal to $f(\big[ (S',l') \big])$ and it follows that $f$ is well-defined. That $f$ is continuous follows from the fact that since $I$ is a finite set, the topology on $\mathrm{Sym}(I_+;A)$ coincides with the discrete topology. Finally, we can define an inverse to $f$:
\begin{align}
g: A^I &\rightarrow \mathrm{Sym}(I_+;A) \nonumber \\
(a_1,...,a_n) &\mapsto \big[ (I_+,l: i \mapsto a_i) \big]. \nonumber
\end{align}
That $g$ is well defined is clear, and that it's continuous follows, as before, from the fact that $A^I$ is given the discrete topology. It is straightforward to see that $g \circ f$ is the identity on $\mathrm{Sym}(I_+;A)$, and $f \circ g$ the identity on $A^I$, proving the claim. 

\end{proof}

\vspace{3mm}

\begin{rek} \label{htpyinvar} A pointed map $M \xrightarrow{f} Y$ induces a map $\mathrm{Sym}(M;A) \xrightarrow{\mathrm{Sym}(f;A)} \mathrm{Sym}(Y;A)$ which is given explicitly by the assignment:
$$ \Big[ (S,l) \Big] \mapsto  \Big[ (f(S), x\mapsto \sum_{s \in f^{-1}(x)}l(s)) \Big] \nonumber $$
and, in fact, this allows us to view the infinite symmetric product as a functor $$\mathrm{Sym}(-;A): \mathrm{Top}_* \rightarrow \mathrm{Top}_*.$$ In the future we will use the notation $\Sym(f;A)\big[(S,l)\big]=\big[f(S),f(l)\big]$. Moreover, this functor is homotopy invariant. Let $H: M\times [0,1]\rightarrow Y$ be a homotopy between two pointed maps $f,g: M \rightarrow Y$ and for each element $t \in [0,1]$, let $H_t$ denote the map $H(-,t):M \rightarrow Y$, then one can define a homotopy between $\mathrm{Sym}(f;A)$ and $\mathrm{Sym}(g;A)$ explicitly as 
\begin{align} H': \mathrm{Sym}(M;A) \times [0,1] &\rightarrow \mathrm{Sym}(Y;A) \nonumber \\
((S,l),t) &\mapsto \mathrm{Sym}(H_t;A)(S,l). \nonumber 
\end{align}
\end{rek}

\vspace{3mm}
\subsection{The relevant categories}
From here on we will fix $M$ to be a smooth, pointed $n$-manifold. 
The following categories will come up throughout our proof:

\begin{defi} Let $\mathrm{Mfld}_n$ denote the category whose objects are $n$-manifolds and with morphisms given by open embeddings.
\end{defi} 
\begin{defi}
Let $\mathrm{Disk} \subset \mathrm{Mfld}_n$ denote the full subcategory consisting of objects which are finite disjoint unions of $n$-dimensional Euclidean spaces. 
\end{defi}
\noindent Of course, the category $\mathrm{Disk}$ also depends on $n$, but since we have fixed a dimension $n$ we will omit that from the notation for ease of reading. We can then consider the over-category $\mathrm{Disk}_{/M}$.

\begin{defi} Let $\mathrm{Disk}_{*/M}$ denote the full subcategory of $\mathrm{Disk}_{/M}$ consisting of objects $U \hookrightarrow M$ whose image contains the basepoint $* \in M$. 
\end{defi}

\begin{defi} Let $\mathrm{Disk}_{*/M}^{\mathrm{\leq 2}}$ denote the full subcategory of $\mathrm{Disk}_{*/M}$ consisting of objects $U \hookrightarrow M$ such that $U$ has at most two connected components. 
\end{defi}

\vspace{3mm}
\section{The Proof}
\vspace{3mm}

\begin{prop} \label{main} There is a homotopy equivalence 
$$\mathrm{Sym}(M;A) \simeq \underset{U \hookrightarrow M \in \mathrm{Disk}_{*/M} }{\mathrm{hocolim}} \mathrm{Sym}(U;A). $$
\end{prop}

\vspace{3mm} 
Before proving this proposition, we recall the following theorem which can be found in greater generality in [Lur16--A.3] :
\begin{thm}[Lurie] \label{LurieSVK} Let $X$ be a topological space and let $\mathscr{U}(X)$ denote the poset of open subsets of $X$. Let $\mathscr{C}$ be a small category and let $ \chi: \mathscr{C} \rightarrow \mathscr{U}(X)$ be a functor. For every $x \in X$ let $\mathscr{C}_x$ be the full subcategory of $\mathscr{C}$ spanned by those objects $C \in \mathscr{C}$ such that $x$ is in the image $\chi(C)$, and suppose the following condition holds: 
\begin{itemize}
\item{For all $x \in X$, the classifying space $\mathrm{B}\mathscr{C}_x$ is contractible.}
\end{itemize}
Then the canonical map $\underset{C \in \mathscr{C}}{\mathrm{hocolim  }} (\chi(C)) \rightarrow X$ is an equivalence. \end{thm}

\begin{proof}[Proof of Proposition \ref{main}]
Observe first that the functor $$\Sym(-;A): \mathrm{Top}_* \rightarrow \mathrm{Top}_*$$ preserves open embeddings. 

To see it this, suppose $g:X \rightarrow Y$ is an open embedding. We'll first show $\Sym(g):=\Sym(g;A): \Sym(X;A) \rightarrow \Sym(Y;A)$ is open. Note that for any finite set $I$, we have the commutative diagram 
\[
\begin{tikzcd}
X^{I_+}\times A^I \ar{r}{f_{I,X}} \ar[swap]{d}{g^I \times {\rm{id}}^I} & \Sym(X;A) \ar{d}{\Sym(g)} \\
Y^{I_+}\times A^I \ar[swap]{r}{f_{I,Y}} & \Sym(Y;A) 
\end{tikzcd}
\]
where $f_{I,X}$ and $f_{I,Y}$ are the maps topologizing $X$ and $Y,$ respectively. For any open $U \subset \Sym(X;A)$, we want to see that $\Sym(g)(U)$ is open in $\Sym(Y;A)$. We know by definition that $f_{I,X}^{-1}(U)$ is open, and because $g$ is an open embedding, $\bigl( g^I \times {\rm{id}}^I \bigr) \circ \bigl( f_{I,X}^{-1}(U) \bigr)$ is also open. Now 

$$\Sym(g)(U)=\underset{I\text{--finite}}{\bigcup} f_{I,Y} \big( (g^I \times \mathrm{id}^I) \circ (f_{I,X}^{-1}(U)) \big)$$ 
\noindent so that $\Sym(g)(U)$ is open as desired.



Now to see $\Sym(g)$ is an embedding, it suffices to see it is an injection. For this, assume the equivalence class $\big[(g(S),g(l))\big]$ is the same as $\big[(g(T),g(m))\big]$, then we will see it follows that $\big[(S,l)\big]$ is the same equivalence class as $\big[(T,m)\big]$. 

From our assumption of the equivalence $\big[(g(S),g(l))\big]=\big[(g(T),g(m))\big]$, we can write without loss of generality the following equalities 
\begin{align}
&g(S)=g(T)\cup g(S')  \\
&g(l)|_{g(T)}=g(m) \\ 
&g(l)|_{g(S')}\equiv e. 
\end{align}
The first line, with the fact that $g$ is an embedding, implies that $T$ is a subset of $S$, so we have the equality $S=T \amalg (S-T).$ Then we will be done if we show that $l|_{T}=m$ and that $l|_{S-T}\equiv e,$ because then the equalities
 $\big[(S,l)\big]=\big[(T \amalg (S-T),m \amalg e)\big]=\big[(T,m)\big]$ will immediately follow. Let us first consider $l|_{T}$, the case of $l|_{S-T}$ will be similar. We know by assumption (3.2) that $g(l)|_{g(T)}=g(m)$ which implies that that for $t$ an element of $T$, we have the string of equalities
 \begin{align}
 m(t)&=\sum_{x \in g^{-1}(g(t))}m(x) \nonumber \\
 &=g(m)(g(t)) \nonumber \\
 &=g(l)(g(t)) \nonumber \\
 &=\sum_{x \in g^{-1}(g(t))}l(x) \nonumber \\
 &=l(t) \nonumber
 \end{align}
 where the first and last line come from the fact that $g$ is an embedding, the second and fourth lines are by definition, and the third line is our assumption (3.2). This shows $l|_T=m$ as desired. To see the equivalence $l|_{S-T} \equiv e$, note by our assumption (3.3), as above we have the following equalities for $s$ an element of $S-T$,
 \begin{align}
 e &= g(l)(g(s))\nonumber \\
 &=\sum_{x \in g^{-1}(g(s))}l(x)  \nonumber \\
 &=l(s) \nonumber
\end{align}
which proves the claim, and further concludes the proof that $\Sym(-;A)$ preserves open embeddings. 

The above discussion then allows us to consider the functor $$\Sym(-;A): \mathrm{Disk}_{*/M}  \rightarrow \mathscr{U}(\mathrm{Sym}(M;A))$$ which sends an object $\iota: U\hookrightarrow M$ to $\mathrm{im}(\mathrm{Sym}(\iota;A))=\mathrm{Sym}(U;A)$ viewed as a subspace of $\mathrm{Sym}(M;A).$  To apply Theorem \ref{LurieSVK}, it remains to check that for a configuration $x~=~(\{*,x_1,...,x_m\},l)$ in $\mathrm{Sym}(M;A)$, the category $(\mathrm{Disk}_{*/M})_x$ has a contractible classifying space. This will follow from the fact that the nerve of a cofiltered category is contractible. To see that $(\mathrm{Disk}_{*/M})_x$ is cofiltered, suppose $f: U\hookrightarrow M$ and $g: V \hookrightarrow M$ are objects in $(\mathrm{Disk}_{*/M})_x$. The intersection $f(U) \cap g(V) \subset M$ is open and non-empty; in particular, it contains at least the set of points $\{*,x_1,...,x_m\}$. It follows by local contractibility that there is a disjoint union of open contractible neighborhoods, one around each element of $\{*,x_1,...,x_m\}$, including into both $f(U)$ and $g(V)$, and which can be realized as an embedding $\amalg_{i=1}^m \mathbb{R}^n \hookrightarrow M$ giving the cofiltration condition. 

We can now apply Theorem \ref{LurieSVK} to see there is an equivalence
$$\underset{U \hookrightarrow M\in \mathrm{Disk}_{*/M}}{\mathrm{hocolim}} \mathrm{Sym}(U;A) \xrightarrow{\simeq} \mathrm{Sym}(M;A).$$

\end{proof}

\begin{rek} In fact, the above homotopy colimit is equivalent to the factorization homology of $M$ with coefficients in the $n$-disk algebra in spaces (as in [AF15-3.2]) determined by the functor $\mathrm{Sym}(-;A)$, which motivates the title of this note.
\end{rek}
\vspace{3mm}

\begin{cor} \label{endtop} There are homotopy equivalences of spaces
$$\mathrm{Sym}(M;A) \simeq \hocolim \Bigl( \mathrm{Disk}_{*/M}\overset{\Sym(-;A)}{\longrightarrow} \mathrm{Top} \Bigr) \simeq \hocolim \Bigl( \mathrm{Disk}_{*/M}\overset{\widetilde{H}_0(-;A)}{\longrightarrow} \mathrm{Top} \Bigr) $$ 
where, for $U \hookrightarrow M$ in $\mathrm{Disk}_{*/M}$, we again consider $\widetilde{H}_0(U;A) \cong A^{|\pi_0(U)|-1}$ as a discretely topologized space. 
\end{cor}

\begin{proof}
For $I$ a finite set, we've seen in Lemma \ref{finite} the homeomorphism $\mathrm{Sym}(I_+;A) \cong A^{I}$. Thus, since for any $U \hookrightarrow M$ in $\mathrm{Disk}_{*/M}$ there is a homotopy equivalence $U \simeq \pi_0(U)$, the corollary follows from the above proposition combined with the homotopy invariance [Remark \ref{htpyinvar}] of $\mathrm{Sym}(-;A)$.
\end{proof}

\vspace{3mm}
Now we wish to compute this homotopy colimit in the category of connective chain complexes. Observe that the functor $\widetilde{H}_0(-;A): \mathrm{Disk}_{*/M} \rightarrow \mathrm{Top}_*$ factors through connective chain complexes in the following fashion:
\[
\begin{tikzcd}
\mathrm{Disk}_{*/M} \ar{rr} \ar[swap]{dr}{\widetilde{H}_0(-;A)[0]; \text{ } d=0} & &\mathrm{Top} \\
&\mathrm{Ch}_{\geq0} \underset{\mathrm{DK}}{ \simeq} \mathrm{sAb} \ar[swap]{ur}{U}&
\end{tikzcd}
\]
where $\mathrm{DK}: \mathrm{Ch}_{\geq0} \rightarrow \mathrm{sAb}$ is part of the Dold--Kan correspondence, $U: \mathrm{sAb} \rightarrow \mathrm{Top}$ is the forgetful functor. We will henceforth simply use $\widetilde{H}_0(-;A)$ to denote the functor $\widetilde{H}_0(-;A)[0]: \mathrm{Disk}_{*/M} \rightarrow \mathrm{Ch}_{\geq 0}.$\\


\vspace{3mm}
\begin{lem} \label{john}
The natural map $$\hocolim \Bigl( \mathrm{Disk}_{*/M}\overset{\widetilde{H}_0(-;A)}{\longrightarrow} \mathrm{Ch}_{\geq 0} \Bigr) \xrightarrow{\simeq} \hocolim \Bigl( \mathrm{Disk}_{*/M}\overset{\widetilde{H}_0(-;A)}{\longrightarrow} \mathrm{Top} \Bigr) $$ is an equivalence.
\end{lem}

\begin{proof}
By inverting isotopy equivalences in $\mathrm{Disk}_{/M}$ and $\mathrm{Disk}_{*/M}$, one obtains $\infty$-categories which we'll denote by $\mathscr{D}\mathrm{isk}_{/M}$ and $\mathscr{D}\mathrm{isk}_{*/M}$, respectively. For details, see [AF15, AFT17]. What is important to us about these $\infty$-categories is that they are homotopy sifted, and the forgetful functor $U$ preserves homotopy sifted homotopy colimits. These categories fit in to the following diagram:

\[
\begin{tikzcd}
\mathscr{D}\mathrm{isk}_{*/M} \ar{r}  &\mathscr{D}\mathrm{isk}_{/M} \\
 \mathrm{Disk}_{*/M} \ar{u} \ar{r} &\mathrm{Disk}_{/M} \ar{u}
\end{tikzcd}
\]

\noindent where the top horizontal functor is full and final and the right vertical functor is a localization in the sense of [Lur09--5.2.7] and hence final [AFT17--2.21]. In this particular case, that is sufficient to imply the left vertical functor is also a localization and hence final, so there is an equivalence $\underset{\mathrm{Disk}_{*/M}}{\hocolim}\ \widetilde{H}_0(-;A) \simeq \underset{\mathscr{D}\mathrm{isk}_{*/M}}{\mathrm{hocolim}}\ \widetilde{H}_0(-;A) $ which, with the siftedness of $\mathscr{D}\mathrm{isk}_{*/M}$ [AF15--3.3.2], proves the lemma.

\end{proof}

\vspace{5mm}

To summarize, we have shown the equivalences 
\begin{align*}
\pi_i(\Sym(M;A)) &\cong \pi_i \bigl(\hocolim \bigl( \mathrm{Disk}_{*/M}\overset{\widetilde{H}_0(-;A)}{\longrightarrow} \mathrm{Top} \bigr) \bigr) \\
& \cong H_i\big(\hocolim \bigl( \mathrm{Disk}_{*/M}\overset{\widetilde{H}_0(-;A)}{\longrightarrow} \mathrm{Ch}_{\geq 0} \bigr) \bigr)
\end{align*}
\noindent where the first equivalence follows from Proposition \ref{main} and Corollary \ref{endtop}, and the second from Lemma \ref{john}. Now all that remains is to see that this does, in fact, agree with the singular reduced homology of $M$.

\vspace{3mm}

\begin{cor}
There is an equivalence of chain complexes $$\Bigl(\hocolim \bigl( \mathrm{Disk}_{*/M}\overset{\widetilde{H}_0(-;A)}{\longrightarrow} \mathrm{Ch}_{\geq 0} \bigr) \Bigr) \simeq \Bigl(\hocolim \bigl( \mathrm{Disk}^{\leq 2}_{*/M}\overset{\widetilde{H}_0(-;A)}{\longrightarrow} \mathrm{Ch}_{\geq 0} \bigr) \Bigr).$$
\end{cor}

\begin{proof}
The corollary follows from the fact that the functor $\widetilde{H}_0(-;A): \mathrm{Disk}_{*/M} \rightarrow \mathrm{Ch}_{\geq 0}$ is the homotopy left Kan extension of its restriction to the subcategory $\mathrm{Disk}_{*/M}^{\leq 2}$. To see this, let $\iota$ denote the inclusion $\mathrm{Disk}_{*/M}^{\leq 2} \xrightarrow{\iota} \mathrm{Disk}_{*/M}$ so there is a natural transformation $$ \mathrm{hLKan}_{\iota}\widetilde{H}_0(-;A) \rightarrow \widetilde{H}_0(-;A) $$ which is given on objects $f: U \hookrightarrow M \in \mathrm{Disk}_{*/M}$ by the canonical maps
$$\mathrm{hLKan}_{\iota}\widetilde{H}_0(-;A)(f)=\underset{ \star \in \iota_{/f}}{\mathrm{hocolim}} \ \widetilde{H}_0(s(\star);A)\rightarrow \widetilde{H}_0(U;A)$$ 
where the homotopy colimit is taken over the over-category $\iota_{/f}$ and $s(\star)$ denotes ``source" of such an diagram $\star$. 

That is, every object in $\iota_{/f}$ is a commutative triangle, $\star$, of the form
\[
\begin{tikzcd}
V  \ar{rr}  \ar{dr} & \MySymb{d} & U \ar{dl}{f} \\
&M &\\
\end{tikzcd}
\]
where $\pi_0(V)$ is at most 2 and we define $s(\star)$ to be $V$.

Now it suffices to check these maps are equivalences on each object. Let $U= \amalg_{i=1}^{m} \mathbb{R}^n_i$ and assume the basepoint of $M$ is contained in $f(\mathbb{R}^n_1)$. Denote by $\mathscr{C}_{/f}$ the subcategory of $\iota_{/f}$ consisting of the $m$ objects of the form
\[
\begin{tikzcd}
\mathbb{R}^n_1 \cup \mathbb{R}^n_{i} \ar[hook]{rr} \ar{dr} & &U \ar{dl}{f}  \\
&M
\end{tikzcd}
\]
where the hooked arrow denotes the canonical inclusion $\mathbb{R}^n_1 \cup \mathbb{R}^n_{i} \hookrightarrow \amalg_1^m \mathbb{R}^n_i$, and with morphisms, only the $(m-1)$ canonical inclusions of the diagram with source $\mathbb{R}_1 \cup \mathbb{R}_1= \mathbb{R}_1$ into any of the other diagrams of the above form. 

In a moment will show the inclusion $F: \mathscr{C}_{/f} \hookrightarrow \iota_{/f}$ is homotopy final, but let us momentarily assume this fact and complete the proof of the corollary, since from here the result is almost immediate. Showing finality allows us to take the homotopy colimit over $\mathscr{C}_{/f}$ instead of $\iota_{/f}$. Now the homotopy colimit $ \underset{ \star \in \mathscr{C}_{/f}}{\mathrm{hocolim}} \ \widetilde{H}_0(s(\star);A)$ is precisely the homotopy colimit of the diagram
$$
\begin{tikzcd}
&\indent \indent \widetilde{H}_0(\mathbb{R}^n_1 \cup \mathbb{R}^n_2) \\
&\indent \indent  \widetilde{H}_0(\mathbb{R}^n_1 \cup \mathbb{R}^n_3) \\
\widetilde{H}_0(\mathbb{R}^n_1) \ar{uur} \ar{ur} \ar{r} \ar{dr} &  \vdots \\
&\indent \indent \widetilde{H}_0(\mathbb{R}^n_1 \cup \mathbb{R}^n_m)
\end{tikzcd}
$$
which is exactly the homotopy coproduct (equivalently, the coproduct) $\coprod_{i=2}^m \widetilde{H}_0(\mathbb{R}_1^n \cup \mathbb{R}^n_i)$.  Finally, since we are working in the category of connective chain complexes, the coproduct is the same as the product so this is equivalent to $\widetilde{H}_0(U;A)$. It follows that the map is an isomorphism as desired:
$$\mathrm{hLKan}_{\iota}\widetilde{H}_0(-;A)(f) \cong \widetilde{H}_0(U;A) \xrightarrow{\sim} \widetilde{H}_0(U;A).$$

Now all that remains to complete the proof of the corollary is to show that the inclusion $F: \mathscr{C}_{/f} \hookrightarrow \iota_{/f}$ is homotopy final. To see this we must show that for any $\star \in \iota_{/f}$ the undercategory $F^{\backslash \star}$ is non-empty and contractible. We can write $\star$ as a diagram:
\[
\begin{tikzcd}
V_1 \amalg V_i \ar{rr}{g} \ar[swap]{dr}{h} & \MySymb{d}  &U \ar{dl}{f}  \\
&M
\end{tikzcd}
\]
where $V_1$ embeds into $\mathbb{R}_1^n$ (so as to hit the basepoint) and $V_i$ embeds into $\mathbb{R}_i^n$ for some $i$. Note we allow for the possibility that $i=1$, or that $V_i$ is empty. 

It's not hard to see that an object of the category $F^{\backslash \star}$ is uniquely determined by an embedding (dashed) making the following diagram commute:
$$
\begin{tikzcd}
 & \mathbb{R}_1 \cup \mathbb{R}_j  \arrow[dd]  \ar[hook]{dr} &  \\
V_1 \cup V_i \arrow[rr, crossing over] \ar[swap]{dr} \ar[swap, dashed]{ur}&   &U \ar{dl}  \\
&M&
\end{tikzcd}
$$
This category is always non-empty, and has an initial object, namely $g: V_1 \cup V_i \hookrightarrow \mathbb{R}^n_1 \cup \mathbb{R}^n_i$ (in fact, this is the only object if $i \neq 1$), and so is contractible, proving the claim and, thus the corollary.
\end{proof}

 \vspace{5mm}
\begin{lem}
The canonical map of chain complexes $$\Bigl(\hocolim \bigl( \mathrm{Disk}^{\leq2}_{*/M}\overset{\widetilde{H}_0(-;A)}{\longrightarrow} \mathrm{Ch}_{\geq 0} \bigr) \Bigr)  \xrightarrow{\simeq} \widetilde{C}_*(M;A)$$ is a quasi-isomorphism. 
\end{lem}
\begin{proof}
Another application of Theorem \ref{LurieSVK} confirms there is an equivalence $M \simeq \underset{ U \hookrightarrow M \in \mathrm{Disk}_{*/M}^{\leq 2}  } {\mathrm{hocolim}} U$, and thus we have the following quasi-isomorphisms:
 \begin{align}
 \widetilde{C}_*(M;A) &=\widetilde{C}_*\Big(\underset{ U \hookrightarrow M \in \mathrm{Disk}_{*/M}^{\leq 2}  }{\mathrm{hocolim}} U;A\Big) \nonumber \\
 &\simeq \underset{ U \hookrightarrow M \in \mathrm{Disk}_{*/M}^{\leq 2}  }{\mathrm{hocolim}}\widetilde{C}_*(U;A) \nonumber \\
 &\simeq \underset{ U \hookrightarrow M \in \mathrm{Disk}_{*/M}^{\leq 2}  }{\mathrm{hocolim}} \widetilde{C}_*(\pi_0(U);A) \nonumber
 \end{align}
where the last equivalence comes from our assumption that $U$ is homotopy equivalent to $\pi_0(U)$.

Now there is an obvious map $$\underset{ U \hookrightarrow M \in \mathrm{Disk}_{*/M}^{\leq 2}}{\mathrm{hocolim}}\widetilde{H}_0(U;A) \rightarrow \underset{ U\hookrightarrow M \in \mathrm{Disk}_{*/M}^{\leq 2}}{\mathrm{hocolim}} \widetilde{C}_*(\pi_0(U);A)$$ which induces an isomorphism on homology, and the lemma follows.
\end{proof}

\vspace{5mm}
Stringing the above together, we arrive immediately at the Dold--Thom theorem:
\begin{thm}[Dold--Thom] \label{DT} For $M$ a smooth pointed manifold and for every integer $i\in \mathbb{Z}$, there is an isomorphism of groups
$$\pi_i(\mathrm{Sym}(M;A)) \cong \widetilde{H}_i(M;A).$$
\end{thm}


\end{document}